\documentclass{amsart}
\usepackage{amssymb,mathrsfs,MnSymbol}
\usepackage{color}
\usepackage{enumitem}

\def\bN {\mathbf{N}}

\def\bR {\mathbf{R}}

\def\fH {\mathfrak{H}}

\def\cC {\mathcal{C}}
\def\cD {\mathcal{D}}

\def\cF {\mathcal{F}}

\def\cH {\mathcal{H}}

\def\cL {\mathcal{L}}

\def\cP {\mathcal{P}}

\def\cS {\mathcal{S}}

\def\a {{\alpha}}
\def\b {{\beta}}

\def\eps {{\epsilon}}

\def\d {{\partial}}
\def\grad {{\nabla}}
\def\Dlt {{\Delta}}

\def\rstr {{\big |}}

\newcommand{\Tr}{\operatorname{trace}}

\newcommand{\Lip}{\operatorname{Lip}}

\newcommand{\MKd}{\operatorname{dist_{MK,2}}}
\newcommand{\Op}{\operatorname{OP}}

\newcommand{\vp}{\epsilon}

\def\hb {{\hbar}}

\newcommand{\ba}{\begin{aligned}}
\newcommand{\ea}{\end{aligned}}

\newcommand{\be}{\begin{equation}}
\newcommand{\ee}{\end{equation}}

\newcommand{\lb}{\label}

\newtheorem{Thm}{Theorem}[section]
\newtheorem{Rmk}[Thm]{Remark}
\newtheorem{Prop}[Thm]{Proposition}
\newtheorem{Cor}[Thm]{Corollary}
\newtheorem{Lem}[Thm]{Lemma}
\newtheorem{Def}[Thm]{Definition}

\newcommand{\dis}{\mbox{\bf d}}

\begin{document}

\title[Quantum Perturbations]{Time dependent Quantum Perturbations uniform in the semiclassical regime}

\author[F. Golse]{Fran\c cois Golse}
\address[F.G.]{CMLS, \'Ecole polytechnique, CNRS, Universit\'e Paris-Saclay , 91128 Palaiseau Cedex, France}
\email{francois.golse@polytechnique.edu}

\author[T. Paul]{Thierry Paul}
\address[T.P.]{Laboratoire J.-L. Lions, Sorbonne Universit\'e \& CNRS, bo\^\i te courrier 187, 75252 Paris Cedex 05, France}
\email{thierry.paul@upmc.fr}

\begin{abstract}
We present a time dependent quantum perturbation result, uniform in the Planck constant, \begin{color}{black}for perturbations of potentials whose gradients are Lipschitz continuous by potentials whose gradients are only bounded a.e.. Though  this  low regularity of the full potential  is not enough to provide the existence of the classical underlying  dynamics, at variance with the quantum one,  our result shows that the classical limit of the perturbed quantum dynamics remains in a tubular neighbourhood of the classical unperturbed one of  size of  order of the square root of the size of the perturbation.\end{color} We treat both Schr\"odinger and von Neumann-Heisenberg equations.
\end{abstract}

\date{\today}

\maketitle

\hfill {\it in memory of {\bf Arthur Wightman}}
\tableofcontents
\section{Introduction}\label{intro}
Perturbation theory has a very special status in Quantum Mechanics. On one side, it is responsible to most of its more spectacular success, from atomic to nuclear physics. On the other side, it has a very peculiar epistemological status: it was while he was working with Max Born \cite{B} on the Bohr-Sommerfeld quantization of celestial perturbations series, as explicitly stated  by Poincar\'e  in his famous ``M\'emoires" \cite{P}, that Heisenberg went to the idea of replacing the commutative algebra of convolution --- corresponding to multiple multiplications of Fourier series appearing in computations on action-angle variables --- by the famous noncommutative algebra of matrices \cite{H}.

After quantum mechanics was truly settled, perturbation theory took a completely different form, in the paradigm of functional analysis ``\`a la Kato" and appeared then  mostly in the framework of the so-called  Rayleigh-Schr\" odinger series. A kind of paradox is that it took a long time to link back the Rayleigh-Schr\" odinger series to the ``original"  formalism of quantization of, say, Birkhoff series \cite{Bi}, though, in the mean time, the latter continued   to be extensively used for applied purpose e.g. in heavy chemical computations.

It seems that Arthur Wightman proposed to several PhD students to work on this problem. One of the difficulty is that, starting with the second term of the Rayleigh-Schr\" odinger expansion, 
$$
E^2_i=\sum\limits_{k}\frac{
\langle\psi^0_i,V\psi^0_k\rangle\langle\psi^0_k,V\psi^0_i\rangle}{E^0_i-E^0_k},
$$
there appears formally poles at zero in the Planck constant, for example when the unperturbed Hamiltonian is  the harmonic oscillator with unperturbed  eigenvalues $E^0_i=(i+~\tfrac12)\hbar$.  Although this poles disappear at the classical limit $\hbar\to 0$ because the sum $\sum\limits_{k}\frac{
\langle\psi^0_i,V\psi^0_k\rangle\langle\psi^0_k,V\psi^0_i\rangle}{i-k}$ vanishes in this limit for parity reasons, controlling all the terms of the series remained for years a task considered as unachievable.

To our knowledge, the first proof on the convergence term by term of the Rayleigh-Schr\" odinger expansion to the quantized Birkhoff one, for perturbations of non-resonant  harmonic oscillators, was given in \cite{GP1}, by implementing the perturbation procedure in the so-called Bargman representation (see also \cite{DGH} for an implementation in the framework of the Lie method). The reader interested in this subject can also consult \cite{Dyn,Dyn2} for a proof (also for general non harmonic unperturbed Hamiltonians) in a generalization of \'Ecalle's mould theory and \cite{NPST} for a link between Rayleigh-Schr\" odinger expansion and Hopf algebras.

\vskip 1cm

When one considers time dependent perturbation theory, i.e. comparison between two quantum evolution associated to  two ``close" Hamiltonians $H,H'$, the situation is more difficult. The simple Duhamel formula
$$
e^{-i\frac{tH}\hbar}-e^{-i\frac{tH'}\hbar}
=
\frac1{i\hbar}\int_0^te^{-i\frac{(t-s)H}\hbar}(H'-H)e^{-i\frac{sH'}\hbar}ds
$$
shows clearly that a pole at zero in the Planck constant is again involved. But  to our knowledge,  no combinatorics or normal form can help to remove it in general and one is usually reduced to the trivial estimate
$$
\|e^{-i\frac{tH}\hbar}-e^{-i\frac{tH'}\hbar}\|
\leq
t\frac {\|H'-H\|}{\hbar}
$$
valid for, e.g. any Schatten norm, the operator, Hilbert-Schmidt  or trace norm for example.

\vskip 1cm

In the present paper, we will get rid of this pole in $\hbar$ phenomenon by estimating the difference between two quantum evolutions (in a 
weak topology consisting in tracing against  a set of test observables)
in two forms:

- one  linear in the norm of the difference of the Hamiltonians plus a term vanishing with $\hbar$

- the other proportional to  the norm of the difference of the Hamiltonians to the power $1/3$ and independent of $\hbar$.

The proofs of our  results, Sections  \ref{proofmain}, \ref{proffcor} and \ref{proofcrasprop}, will be using the framework of the von Neumann-Heisenberg equation for density operators $D$, 
$$\partial_tD=\tfrac1{i\hbar}[D,H],
$$
 but our results, Theorem  \ref{main1} and Corollary \ref{main2}, will be first presented for pure states, Section \ref{results}, that is when $D=|\psi\rangle\langle\psi|$, in which case it reduces to the usual Schr\"odinger equation (modulo a global phase of the wave function)
$$
i\hbar\partial_t\psi= H\psi.
$$
The mixed states situation will be treated in Section \ref{resultmix}, Theorem \ref{main1mix}.

Our results will need very low regularity of the perturbed potential, namely the boundness of its gradient, and of the unperturbed one, Lipschitz continuity of its gradient. In this situation, the classical  underlying dynamics is well posed for the unperturbed Hamiltonian, but not for the perturbed one. To our knowledge, the classical limit for pure state in this perturbed situation is unknown. We show, in Section \ref{appclass}  Theorem \ref{main3}, that the limit as $\hbar\to 0$ of the Wigner function of the wave function at time $t$ is  close to the one of the initial state pushed forward by the unperturbed classical flow.
\newcommand{\vd}{U}
\newcommand{\vu}{V}
\section{Main result}\label{results}

For $\lambda,\mu\in[0,1]$, let us consider the quantum Hamiltonian 
$$
\cH^{\lambda,\mu}_0=\cH_0:=-\tfrac12\hb^2\Dlt_x+\tfrac\lambda2|x|^2+\mu \vu
$$
on $\fH:=L^2(\bR^d)$. Here $\vu\equiv \vu(x)\in\bR$  such that $\vu\in C^{1,1}(\bR^d)$. For any other real potential $\vd\in W^{1,\infty}(\bR^d)$, we define, for $\eps\in[0,1]$
$$
\cH^{\lambda,\mu}_\eps=\cH_\eps:=-\tfrac12\hb^2\Dlt_x+\tfrac\lambda2|x|^2+\mu \vu+\eps \vd.
$$
Henceforth we denote 
\begin{eqnarray}
\cH&:=&\cH_0^{1,0}=-\tfrac12\hb^2\Dlt_x+\tfrac12|x|^2  \mbox{ \hfill (harmonic oscillator)}\nonumber\\
\cD(\fH)&:=&\{R\in\cL^1(\fH)\text{ s.t. }R=R^*\ge 0\text{ and }\Tr_\fH(R)=1\}\, \mbox{ (density operators)}, 
\nonumber\\
\cD_2(\fH)&:=&\{R\in\cD(\fH)\text{ s.t. }\Tr_\fH(R^{1/2}{\cH} R^{1/2})<\infty\}\,\mbox{ (
finite second moments)}.
\nonumber
\end{eqnarray}

For $\psi\in\fH$, we define
\be\label{defdeltapsi}
\Delta(\psi)
:=\sqrt{\left(\psi,(x-(\psi,x\psi))\psi)^2+(-i\hbar\nabla_x-(\psi,-i\hbar\nabla_x\psi))^2\psi\right)}
\ee
Note that the Heisenberg inequalities 
$$
\sqrt{
(\psi,
(x_k-(\psi,x_k\psi))^2,\psi)}
\sqrt{(\psi,(-i\hbar\nabla_{x_k}-(\psi,-i\hbar\nabla_{x_k}\psi))^2
\psi)}\geq\hbar/2,\ k=1,\dots,d
,
$$
imply that
$$
\Delta(\psi)\geq\sqrt{2d\hbar}.
$$

On $\cD(\fH)$ we define the following distance 
\be\label{defd}
\dis(R,S):=
\sup_{\substack{
\max\limits_{\substack{|\a|,|\b| \le 2[\frac d4]+3}}\|\mathcal D^\a_{-i\hbar\nabla}\mathcal D^\b_{x}F\|_{1}\leq 1}}|\Tr{(F(R-S))}|\,,
\ee
where $\cD_A=\tfrac1{i\hbar}[A,\cdot]$ for each (possibly unbounded) self-adjoint operator $A$ on $\fH$ and $\|.\|_1$ is the trace norm on $\cD(\fH)$.
The fact that $\dis$ is a distance has been proved in \cite[Appendix A]{FGJinPaul}.

Abusing the notation, for $\phi,\psi\in\mathfrak{H}$, we set
for $\psi,\varphi\in\fH$
$$
\dis(\psi,\varphi):=\dis(|\psi\rangle\langle\psi|,|\varphi\rangle\langle\varphi|)
$$
\vskip 1cm
Consider the family of Schr\"odinger equations, for $\eps\in[0,1]$,
\be\lb{vNeps}
i\hb\d_t\psi_\eps(t)=\cH
_\eps\psi_\eps(t),\,\qquad \psi_\eps(0)=\psi^{in}_\eps\in
H^2(\bR^d)\,.
\ee

\begin{Thm}\label{main1}
Let $\psi^{in}_\eps$ satisfy the  following hypothesis:
\be\label{idelta} 
\Delta(\psi^{in})=O(\sqrt{\hbar}).
\ee Then, for every $t$, 
$$
\dis(\psi_0(t),\psi_\eps(t))^2\leq C(t)\epsilon+D(t)\hbar,
$$
where $C(t),D(t)$, given by \eqref{cdet},\eqref{ddet}, satisfy
\begin{eqnarray}
C(t)&=&\frac{e^{|t|(1-\lambda+\mu\Lip(\nabla \vu))}-1}{1-\lambda+\mu\Lip(\nabla \vu)}
C_{
(\psi_\eps^{in},\cH\psi_\eps^{in}),\|\vu\|_\infty,\|\vd\|_\infty,\|\nabla \vd\|_\infty}<\infty\nonumber\\
D(t)&=&e^{|t|(1-\lambda+\mu\Lip(\nabla \vu))}
D_d
<\infty\nonumber
\end{eqnarray}
\end{Thm} 

The following result gives an upper bound independent of $\hbar$.

\begin{Cor}\label{main2}
Under the same assumptions as in Theorem \ref{main1}, 
$$
\dis(\psi_0(t),\psi_\eps(t))\leq E(t)\epsilon^{\frac13},
$$
with
$$
E(t)=\min{(\sqrt{C(t)+D(t}),2|t|\|\vd\|^2_\infty)}.
$$
\end{Cor}
The function $z\mapsto\frac{e^z-1}z$ is extended by continuity at $z=0$, so that, when $\lambda=1, \ \Lip(\nabla V)=0$ (perturbation of the harmonic oscillator), $C(t)$ increases linearly in time and $D(t)$ is independent of time and $C(t),E(t)$ increase linearly in time.

\begin{Rmk}
Other choices than the hypothesis 
\ref{idelta} are possible, that we didn't mention for sake of clarity of the main statements. For example

\ref{idelta}'\centerline{  $\Delta(\psi^{in})=O(\sqrt\eps)$: }
 in this case the statement of both Theorem \ref{main1} and Corollary \ref{main2} remain the same with a slight chance of the constants $C(t),D(t)$.
 
 \ref{idelta}''\centerline{  $\Delta(\psi^{in})=O(\hbar^\alpha),\ 0\leq\alpha<1/2$: }
 in this case the statement of both Theorem \ref{main1} and Corollary \ref{main2} become
 \begin{eqnarray}
\dis(\psi_0(t),\psi_\eps(t))^2&\leq& C(t)\epsilon+D'(t)\hbar^\alpha\nonumber\\
\dis(\psi_0(t),\psi_\eps(t))&\leq& E'(t)\epsilon^{\frac\alpha{\alpha+2}}\nonumber
\end{eqnarray}
for constants $C(t),D(t),E(t)$ easily computable from the proofs of Section \ref{proofmain}.
\end{Rmk}
\vskip 1cm
Let us finish this section by some topological remarks, inspired by \cite[Section 4]{gpcoh}. The distance $\dis$ defines a weak topology, very different a priori of the usual strong topologies associated to Hilbert spaces in quantum mechanics. Nevertheless, it seems to us better adapted to the semiclassical approximation for the following reason.

Let us consider two coherent states pinned up at two points $z_1=(p_1,q_1),z_2=(p_2,q_2)$ of the phase-space $T^*\bR^d$: $\psi_{z_j}(x)=(\pi\hbar)^{-d/4}e^{-i\tfrac{p_j.x}\hb}e^{-\tfrac{(x-q_j)^2}{2\hbar}},\ j=1,2$.

An easy computation shows that
$$
\|\psi_{z_1}-\psi_{z_2}\|^2_{L^2(\bR^d)}
=
\||\psi_{z_1}\rangle\langle\psi_{z_1}|
-
|\psi_{z_2}\rangle\langle\psi_{z_2}|
|_{Hilbert-Schmidt}^2
=
1-e^{-\frac{|z_1-z_2|^2}{2\hbar}}
$$
so that, as $\hbar\to 0$,
\begin{eqnarray}
\|\psi_{z_1}-\psi_{z_2}\|^2_{L^2(\bR^d)}
=
\||\psi_{z_1}\rangle\langle\psi_{z_1}|
-
|\psi_{z_2}\rangle\langle\psi_{z_2}|
|_{Hilbert-Schmidt}^2
&=& 0\ \ \mbox{ if } z_1= z_2\nonumber\\
&\to&1\ \ \  \forall\ z_1\neq z_2.\nonumber
\end{eqnarray}
In other words, the Lebesgue or Schatten norms behave for small values of $\hbar$ as the discrete topology, the one which only discriminates points.

On the contrary, $\dis$ is much more sensitive to the localization on phase space as shows our next result, proven in Section \ref{proofcrasprop} below.

\begin{Prop}\label{crasprop}
For any bounded convex domain $\Omega\subset\bR^{2d}$, there exists $C_\Omega>0$ such that, for any $z_1,z_2\in\Omega$,
$$
C_\Omega|z_1-z_2|-\hbar\leq2^d\dis(\psi_{z_1},\psi_{z_2})
\leq
2^d\sqrt{|z_1-z_2|^2+2d\hbar}+C_d\hbar,
$$
where $C_d$ is defined in Lemma \ref{lemtheor}  Section \ref{proofmain} below.
\end{Prop}

\section{Applications to the classical limit}\label{appclass}
The
estimates provided by the results of the two preceding sections \begin{color}{black}do not require $\grad \vd$  to be 
continuous --- in other words, the classical dynamics underlying the quantum dynamics generated by $\cH_\eps,\ \eps>0$,  fails to satisfy the assumptions of the 
Cauchy-Peano-Arzel\`a
Theorem.\end{color}
 
Let us recall that one way to look at the transition from quantum to classical dynamics as $\hbar\to 0$ is to associate to a quantum (pure or mixed) state, namely a 
positive trace one operator $R^\hb$ on $\fH$ (density operator) with integral kernel $r^\hb(x,x')$, e.g. a pure state $R^\hb=|\psi^\hb\rangle\langle\psi^\hb|$, for any vector $\psi^\hb$ in $\fH$ the 
 so-called Wigner transform defined on phase-space by (with a slight abuse of notation again)
\begin{eqnarray}\label{defwwig}
W_\hb[R^\hb](x,\xi)&:=&\frac1{(2\pi)^d}\int_{\bR^d}e^{-i\xi\cdot y}r^\hb(x+\tfrac12\hb y,x-\tfrac12\hb y)dy\,\\
W_\hb[\psi^\hb](x,\xi)&:=&
\frac1{(2\pi)^d}\int_{\bR^d}e^{-i\xi\cdot y}{\psi^\hb(x+\tfrac12\hb y)}\overline{\psi_\eps(x-\tfrac12\hb y)}dy\,.
\label{defwigps}
\end{eqnarray}
An easy computation shows that $W_\hb[\psi^\hb]$ is linked to $\psi^\hb$ by the two following marginal properties
\begin{eqnarray}
\int_{\bR^d} W_\hb[\psi^\hb](x,\xi)d\xi&=&|\psi^\hb(x)|^2\label{margx}\\
\int_{\bR^d} W_\hb[\psi^\hb](x,\xi)dx&=&|\widehat\psi^\hb(p)|^2,\ \widehat\psi^\hb(p):=\int_{\bR^d}e^{-ip\cdot x/\hbar}\psi^\hb(x)\tfrac{dx}{(2\pi\hbar)^{d/2}}\label{margxi}
\end{eqnarray}

It has been proved, see e.g. \cite{lionspaul}, that, under the tightness conditions
\begin{eqnarray}
\lim_{R \to +\infty} \sup_{\hb\in (0,1)} \int_{\bR^d \setminus B_{R}^{(d)}} r^\hb(x,x) \; dx &=&0,\label{tight1}\\
\lim_{R \to +\infty} \sup_{\hb\in(0,1)} \frac{1}{(2\pi\hb)^d} \int_{\bR^d \setminus B_{R}^{(d)}} \mathcal Fr^\hb\Bigl(\frac{p}{\hb},\frac{p}{\hb}\Bigr) \; dp &=&0,\label{tight2}
\end{eqnarray}
where $B_{R}^{(d)}$ is the ball of radius $R$ in $\bR^d$ and $\mathcal F$ is the Fourier transform on $\bR^{2d}$,
the family of Wigner  functions $W_\hb[R^\hb]$ converges weakly, in particular in $\cS'(\bR^{2d})$, after extraction of a subsequence of values of $\hbar$, to $W_0\in\cP(\bR^{2d})$, the space of probability measures on $\bR^{2d}$. The measure $W_0$ is called the Wigner measure of the family $R_\hb$. 

 As it is quite standard, we will omit to mention the extraction of subsequences, together with the explicit dependence of states in the Planck constant, and we will just write, when it does not create any confusion,
$$
\lim_{\hbar\to 0}W_\hb[R]=W_0.
$$

Note that when $R=|\psi\rangle\langle\psi|$ is a pure state,  \eqref{tight1}-\eqref{tight2} reads
\begin{eqnarray}
\lim_{R \to +\infty} \sup_{\hb\in (0,1)} \int_{\bR^d \setminus B_{R}^{(d)}} |\psi(x)|^2 \; dx &=&0,\label{tight1ps}\\
\lim_{R \to +\infty} \sup_{\hb\in(0,1)} 
 \int_{\bR^d \setminus B_{R}^{(d)}} 
|\widehat\psi(p)|^2 
 \; dp &=&0\label{tight2ps}
\end{eqnarray}

Considering the quantum Hamiltonian $\cH_\eps$, the expected underlying classical dynamics is the one driven by the Liouville equation 
\be\label{liouveps}
\partial_t\rho=\{\tfrac12{(p^2+\lambda q^2)}+\mu \vu(q)+\eps\vd,\rho\},\ \rho^t|_{t=0}=\rho^{in}
\ee
 where $\{.,.\}$ is the Poisson bracket on the symplectic manifold $T^*\bR^d\sim\bR^{2d}$.

When $\eps=0$, the Hamiltonian vector field of Hamiltonian $\tfrac{p^2+\lambda q^2}2+\mu \vu(q)$ is Lipschitz continuous.
Moreover, it was proven, \cite{lionspaul}, that $R_\hb(t)$ is tight for any $t\in\bR$ and $W_0^t:=\lim\limits_{\hbar\to 0}W_\hb[R_\hb(t)]=\rho^t$ solving \eqref{liouveps} with $\rho^{in}=W_0$.

\begin{color}{black}When $\eps>0$, the Liouville equation \eqref{liouveps} exits the 
Cauchy-Peano-Arzel\`a
category: the associated Hamiltonian vector field might fail to have a  characteristic  through every point of the phase-space.\end{color}  Nevertheless, as shown in \cite[Theorem 6.1]{AFFGP}, \eqref{liouveps} is still well posed in $L^{\infty}_+([0,T];L^{1}(\bR^{2d})\cap L^{\infty}(\bR^{2d}))$, and it was proven in \cite{filipa} (after \cite{AFFGP}), that the Wigner function $W_\hb[R_\hb(t)]$ of the solution of the von Neumann equation
\be\label{primovne}
i\hb\d_tR_\eps(t)=[\cH
_\eps,R_\eps(t)]\,,\qquad R_\eps(0)=R^{in}_\eps
\ee
tends weakly
  to the solution of \eqref{liouveps}, under certain conditions on $R_\eps^{in}$.

Unfortunately, these conditions exclude definitively pure states, as, for example,  one of them impose that $\|R_\eps^{in}\|=O(\hbar^d)$ and, to our knowledge, nothing is known  concerning the dynamics of the (possible) limit of $W_\hb[\psi_\eps(t)]$ as $\hbar\to 0$ where $\psi_\eps(t)$ solves the Schr\"odinger equation \eqref{primovne}. 

Our next result will show that such a limit remains $\sqrt\eps$-close to the push-forward of the Wigner measure of the initial condition by the flow of the unperturbed classical Hamiltonian.

\begin{Thm}\label{main3}
Let $R_\eps(t)$ be the solution of the von Neumann equation \eqref{primovne} with $R_\eps^{in}$ satisfying
$$
\Delta(R_\eps^{in})=O(\sqrt\hb).
$$ 
Let $R_\eps^{in}$ be tight, in the sense that it satisfies \eqref{tight1}-\eqref{tight2}, so that $W_\hb[R_\eps^{in}]\to W_0^{in}$, its Wigner measure, as $\hbar\to 0$.

Then, for any $t\in\bR$,
the family $R_\eps(t)$ is tight, so that 
$$
W_\hb[R_\eps(t)]
\to W_0(t)\in\cP(\bR^{2d})
\mbox{ as }\hbar\to 0.
$$
Moreover,
\begin{enumerate}
\item
$W_0(t)$ is $\sqrt\eps$-close to $\Phi^t\#W_0^{in}$, where $\Phi^t$ is the 
flow of Hamiltonian 
$$\tfrac{p^2+(1-\lambda )q^2}2+\mu\vu(q),
$$
 in the sense that
$$
\sup_{\substack{\int_{\bR^d}\sup\limits_{q\in\bR^d}|
\cF_pf(q,z)|dz\leq 1\\
\max\limits_{\substack{|\a|+|\b|\\\leq 2[d/4]+3}}\|\d^\a_q\d^\b_pf\|_{L^\infty(\bR^{2d})}\leq 1
}}
|\int_{\bR^{2d}}f(q,p)\left(W_0(t)-\Phi^t\#W_0^{in}\right)(q,p)dqdp|
\leq 2^{-d}\sqrt{C(t)}\sqrt\eps,
$$
where  $\cF_pf(q,z):=\int_{\bR^d}e^{-ip.z}f(q,p)dp$ and $C(t)$ is as in Theorem \ref{main1} after replacing $(\psi_\eps^{in},\cH\psi_\eps^{in})$ by $\int_{\bR^{2d}}(\tfrac{p^2+q^2}2)W_0(dpdq)$.
\ \\
\item
In particular, $W_0(t)$ is weakly $\sqrt\eps$-close to $\Phi^t\#W_0^{in}$ in the sense of distribution as, for all test functions $\varphi\in\cS(\bR^{2d})$,
$$
|\int_{\bR^{2d}}\varphi(q,p)\left(W_0(t)-\Phi^t\#W_0^{in}\right)(dp,dq)|
\leq C\varphi(t)\sqrt\eps
$$
with 
$$C_\varphi(t)=\max{\left(\int_{\bR^d}\sup\limits_{q\in\bR^d}|
\cF_p\varphi(q,z)|dz,\max\limits_{|\a|+|\b|\leq 2[d/4]+3}\|\d^\a_q\d^\b_p\varphi\|_{L^\infty(\bR^{2d})}\right)}2^{-d}\sqrt{C(t)}.$$
\begin{color}{black}
\item 
Finally, 
$$
\MKd(W_0(t),\Phi^t\#W_0^{in})
\leq \sqrt{2^{-2d-1}C(t)}\sqrt\eps,
$$
where $\MKd$ is the Monge-Kantorovitch-Wasserstein distance  of order two, whose definition is recalled, for example, in \cite[Section 1]{GParma}.
\end{color}
\end{enumerate}

\end{Thm}
\begin{color}{black}
%
\vskip 1cm
To our knowledge, no existence result is known for the Liouville equation associated to a vector field whose components are in $L^\infty(\bR^{2d})$ and not a priori continuous and for general measure initial data - the only known to us result being the existence and uniqueness  in $L^1(\bR^{2d})\cap L^\infty(\bR^{2d})$ result recalled before, excluding concentration to trajectories of the ODE associated to the vector field (note that well posedness for this ODE has been proved in \cite{jabin} under the extra hypothesis  for the vector field to be in $H^{\frac34}$).

Our result includes the case of concentrating initial data, by taking for example a coherent state for $\psi^{in}$ which provides a Dirac mass as Wigner measure: 

\noindent if $\psi(x)=(\pi\hbar)^{-d/4}e^{-i\tfrac{p.x}\hb}e^{-\tfrac{(x-q)^2}{2\hbar}},\ (p,q)\in T^*\bR^d$, one shows easily that.
$$
W_\hb[|\psi\rangle\langle\psi|](x,\xi)=(\pi\hbar)^{-d}e^{-\frac{(x-q)^2+(\xi-p)^2}\hb}\longrightarrow\delta_{(q,p)}\mbox{ as }\hbar\to 0.
$$

The meaning of the Theorem \ref{main3} can be summarized by the following diagram:  one can ``regularizes" the Liouville equation associated to an $L^\infty$ perturbation of a Cauchy-Lipschitz vector field with a Wigner measure for initial data by the associated Heisenberg-von Neumann equation, one remains close to the unperturbed solution.

\begin{center}


\bigskip
\begin{tabular}{ccc}
&{\fbox{\bf ``Schr\"odinger"}}&\\[5mm]
$R_\hb^{in}$& {$\stackrel{0\to t}{\longrightarrow}$}& $R_\hb(t) $
\\ [7mm]
\scriptsize{${\hbar\to 0}$}\ \LARGE{$\downarrow$}\ \ \ \ \ \ \  & & \ \ \ \ \ \ \  \LARGE{{$\mathbf{\downarrow}$}}\ 
\scriptsize{${\hbar\to 0}$} 
\\ [7mm]
 $\begin{array}{rrl}W_0^{in}&:=&\lim\limits_{\hbar\to 0}W_\hb[R_\hb^{in}]\end{array}$&{$\stackrel{0\to t}{\longrightarrow}$}&
 \begin{color}{black}$\begin{array}{rrl}
 W_0(t)&:=&\lim\limits_{\hbar\to 0}W_\hb[R_\hb(t)]\nonumber\\
 &=&\Phi^t_{\eps=0}\#W_0^{in}+O(\sqrt\eps)\nonumber
 \end{array}$\end{color}\\[3mm]
 &{\fbox{\bf ``Liouville"}}& 
\end{tabular}
\bigskip


\centerline{ Semiclassical regularization of \begin{color}{black} rough Liouville equation $\eps$-close to Cauchy-Lipschitz\end{color} }

\bigskip
\bigskip
\end{center}

\end{color}
\begin{proof}
The propagation of tightness is proved as follows.

%
Let $\chi \in C^\infty(\bR^d)$, $0\leq \chi \leq 1$ such that $\chi(x)=0$ if $|x|<1/2$ and $\chi(x)=1$ if $|x|>1$, and define
$\chi_R(x):=\chi(x/R)$. Obviously
$$
\int_{\bR^d \setminus B_{R}^{(d)}} |\psi_\eps(t)(x)|^2 \; dx\leq
\int_{\bR^d }\chi_R(x) |\psi_\eps(t)(x)|^2 \; dx
$$
Moreover, for some $C>0$,
$
\|\nabla \chi_R\|_{\infty} ,\|\Delta \chi_R\|_{\infty} \leq C/R^2
$, and 
$$
-\frac i\hb[\chi_R,\cH_\eps]=-\frac i\hb[\chi_R,-\tfrac{\hbar^2}2\Delta]
=
-i\hbar(\tfrac12\Delta\chi_R-i\nabla\chi_R\cdot\nabla).
$$
Therefore
\begin{eqnarray}
\partial_t
\int_{\bR^d } \chi_R(x)|\psi_\eps(t)(x)|^2 \; dx&=&
-i\hbar\int_{\bR^d }\bar\psi_\eps(t)(x)((\tfrac12\Delta\chi_R-i\nabla\chi_R.\nabla)\psi_\eps(t))(x)dx\nonumber\\
&=&\int_{\bR^d }( -i\hbar\tfrac12\Delta\chi_R(x)|\psi_\eps(t)|^2\nonumber\\
&&+\bar\psi_\eps(t)(x)\nabla\chi_R(x)\cdot(-i\hbar\nabla\psi_\eps(t)(x))dx,\nonumber
\end{eqnarray}
so that
\begin{eqnarray}
\partial_t
\int_{\bR^d } \chi_R(x)|\psi_\eps(t)(x)|^2 \; dx
&\leq&\hbar\frac C{2R^2}+\frac CR\|-i\hbar\nabla\psi_\eps(t)\|_{L^2(\bR^d)}\nonumber\\
&=&\hbar\frac C{2R^2}+2\frac CR
(\psi_\eps(t),\cH_0^{0,0}\psi_\eps(t))_{L^2(\bR^d)}\nonumber\\
&\leq&\hbar\frac C{2R^2}+2\frac CR\left(
(\psi_\eps(t),\cH_\eps^{\lambda,\mu}\psi_\eps(t))_{L^2(\bR^d)}+\mu\|\vu\|_\infty+\eps\|\vd\|_\infty\right)\nonumber\\
&=&\hbar\frac C{2R^2}+2\frac CR\left(
(\psi_\eps^{in},\cH_\eps^{\lambda,\mu}\psi_\eps^{in})_{L^2(\bR^d)}+\mu\|\vu\|_\infty+\eps\|\vd\|_\infty\right)\nonumber
\end{eqnarray}
and finally, for $t\in[0,T]$
\begin{eqnarray}
\int_{\bR^d } \chi_R(x)|\psi_\eps(t)(x)|^2 \; dx
&\leq&
\int_{\bR^d } \chi_R(x)|\psi_\eps^{in}(x)|^2 \; dx\nonumber\\
&&
+\left(\hbar\frac C{2R^2}+2\frac CR\left(
(\psi_\eps^{in},\cH_\eps^{\lambda,\mu}\psi_\eps^{in})_{L^2(\bR^d)}+\mu\|\vu\|_\infty+\eps\|\vd\|_\infty\right)\nonumber\right)T.
\end{eqnarray}
Therefore $\psi_\eps(t)$ satisfies \eqref{tight2ps} as soon as $\psi_\eps^{in}$ does.

Finally, let us remark that
\begin{eqnarray}
\int_{\bR^d \setminus B_{R}^{(d)}} |\widehat\psi(p)|^2 \; dp &\leq&
\frac1{R^2}
\int_{\bR^d }p^2 |\widehat\psi(p)|^2 \; dp \nonumber\\
&=&\frac2{R^2}(\psi_\eps(t),\cH_0^{0,0}\psi_\eps(t))_{L^2(\bR^d)}
\end{eqnarray}
and one concludes the same way.
\vskip 0.5cm
The rest of the Theorem is proved as follows.
\begin{enumerate}
\item One knows from \cite{lionspaul} that the convergence of Wigner functions to Wigner measure as $\hbar\to 0$ takes place in the dual of the set of test functions $f$ on $\bR^{2d}$ satisfying
\be\label{testf}
\int_{\bR^d}\sup\limits_{q\in\bR^d}|
\cF_pf(q,z)|dz<\infty.
\ee
Since $\vu\in C^{1,1}$ one knows that, for such a test function, 
$$
\lim\limits_{\hbar\to 0}\int_{\bR^{2d}}f(x,\xi)W_\hb[\psi_0(t)](x,\xi)dxd\xi=\int_{\bR^{2d}}f(x,\xi)\Phi^t\#W_0(x,\xi)dxd\xi.
$$
On the other hand, we have the slight variant of Theorem \ref{main1}, proven also in Section \ref{proofmain}.
\begin{Prop}\label{propdelta}
Let $\delta$ be defined by \eqref{defdelta} below. Then
$$
\delta(W_\hb[\psi_0(t)],W_\hb[\psi_\eps(t)])\leq 2^{-d}\sqrt{C(t)\epsilon+D(t)\hbar},
$$
where $C(t),D(t)$ are the constants defined in Theorem \ref{main1}.
\end{Prop} 

Proposition \ref{propdelta} tells us that, for any $f$ satisfying 
$$\max\limits_{|\a|+|\b|\leq 2[d/4]+3}\|\d^\a_q\d^\b_pf\|_{L^\infty(\bR^{2d})}\leq 1,$$
$$
|\int_{\bR^{2d}}f(q,p)(W_\hb[\psi_\eps(t)]-W_\hb[\psi_0(t)])(dpdq)|\leq2^{-d}\sqrt{C(t)\eps+D(t)\hbar}
$$
Hence for any $f$ satisfying 
$$\int_{\bR^d}\sup\limits_{q\in\bR^d}|
\cF_pf(q,z)|dz\leq 1,\ \max\limits_{|\a|+|\b|\leq 2[d/4]+3}\|\d^\a_q\d^\b_pf\|_{L^\infty(\bR^{2d})}\leq 1,$$
we have
\begin{eqnarray}
|\int_{\bR^{2d}}f(q,p)(W_\hb[\psi_\eps(t)]-\Phi^t\#W_0(x,\xi)])(dpdq)|&&\nonumber\\
\leq2^{-d}\sqrt{C(t)\eps+D(t)\hbar}
+
|\int_{\bR^{2d}}f(q,p)(W_\hb[\psi_0(t)]-\Phi^t\#W_0(x,\xi))(dpdq)|&&\nonumber
\end{eqnarray}
and we conclude by taking first the supremum on the functions $f$ and then the limit $\hbar\to 0$ on both sides.
\item The proof is obvious by homogeneity.
\item 
\begin{color}{black}
Since functions in the Schwartz class satisfy \eqref{testf}, one knows that, as $\hbar\to 0$,  $W_\hb[\psi_\eps(t)]\to W_0(t)$ and $W_\hb[\psi_0(t)]\to \Phi^t\#W_0^{in}$ both in $\cS'(\bR^{2d})$. Therefore, by \cite[Theorem 2.3. (2)]{FGMouPaul}, one knows  that
$$
\MKd(W_0(t),\Phi^t\#W_0^{in})\le\varliminf_{\hbar\to 0}MK_\hb(|\psi_\eps(t)\rangle\langle\psi_\eps(t)|,|\psi_0(t)\rangle\langle\psi_0(t)|)\,.
$$
By Theorem \ref{main} with $R_0^{in}=R_\eps^{in}=|\psi_0^{in}\rangle\langle\psi_0^{in}|=|\psi_\eps^{in}\rangle\langle\psi_\eps^{in}|$, we have that
$$
MK_\hb(|\psi_\eps(t)\rangle\langle\psi_\eps(t)|,|\psi_0(t)\rangle\langle\psi_0(t)|)
\leq\sqrt{\gamma(t)\eps},
$$
where $\gamma(t)$ is defined in \eqref{defgam}.

We conclude by noticing that, as $\hbar\to 0$, $2^{2d+1}\gamma(t)\to C(t)$ defined in item $(1)$.
\end{color}
\end{enumerate}
\end{proof}
\section{The case of mixed states}\label{resultmix}
 Consider the family of  von Neumann equations, for $\eps\in[0,1]$,
\be\lb{vNeps}
i\hb\d_tR_\eps(t)=[\cH
_\eps,R_\eps(t)]\,,\qquad R_\eps(0)=R^{in}_\eps\in\cD(\fH).\,
\ee

Let $R^{in}_\eps=R^{in}_0=R\in\cD_2(\fH)$ satisfy one of the five following hypothesis:
\begin{enumerate}[label=(\roman*)]
\item\label{idelta} $\Delta(R^{in})=O(\sqrt{\hbar})$ where the standard deviation $\Delta(R)$ is defined in Lemma \ref{lemdelta} below;
\item\label{iwigner} $\sqrt{R^{in}}$ satisfies, for some $C>0$,
$$
\sup_{\substack{|\beta_1|,\dots,|\beta_d|\leq 7}}
 |\prod_{m=1}^dD_{(x,\xi)}^{\beta_m}   W_\hb[\sqrt{R^{in}}](x,\xi)| 
 \leq \frac {C(2\pi\hbar)^{-\frac d2}}{((\xi^2+x^2)^2+d)^{\frac{10}4+3\epsilon}}\hfil\ \ \  \forall (x,\xi)\in\bR^{2d},
$$
where $W_\hb[\sqrt{R^{in}}]$ is the Wigner transform of $R^{in}$;
\item\label{ihermitte} $\sqrt{R^{in}}$ satisfies, for $C
>0$ and all $j\in\bN^{d}$,
\begin{enumerate}
\item \label{iaj}
$|(H_i,
\sqrt{R^{in}}
 H_j)|\leq C
 (2\pi\hbar)^{\frac d2} 
 {\prod\limits_{1\leq l\leq d}|\hbar j_l+\frac12|^{-
\frac34
-\epsilon}(|i_l-j_l|+1)^{-2-\epsilon}}$,

\item\label{ioaj}
$\sup\limits_{O\in\Omega_1}{|(H_i,\tfrac1{i\hbar}[O,
\sqrt{R^{in}}
] H_j)|}\leq C
{(2\pi\hbar)^{\frac d2}}
{\prod\limits_{1\leq l\leq d}|\hbar j_l+\frac12|^{-
\frac12
-\epsilon}(|i_l-j_l|+1)^{-1-\epsilon}}$,

\ where $\Omega_1=\{y_j,
\pm\hbar
\partial_{y_j}\ on\ L^2(\bR^d,dy),\ j=1,\dots,d\}$ and the $H_js$ are the semiclassical Hermite functions;
\end{enumerate}
\item\label{itoplitz}  $R^{in}$ is a T\"oplitz operator;
\item\label{itoplitz2} there exist a T\"oplitz operator $T_F$ such that
$
T_F^{-\frac12}\sqrt{R^{in}}\ ,\sqrt{R^{in}}T_F^{-\frac12}$ and $ 
T_F^{-\frac12}\sqrt{R^{in}}\frac1{i\hbar}[O,\sqrt{R^{in}}T_F^{-\frac12}],\ 
O\in\Omega_1$ are bounded on $L^2(\bR^d)$.
\end{enumerate}
\begin{Thm}\label{main1mix}
For every $t\geq 0$, 
$$
\dis(R_0(t),R_\eps(t))^2\leq C(t)\epsilon+D(t)\hbar,
$$
where $C(t),D(t)$ are  given by \eqref{cdet}-\eqref{ddet}, and are of the form
\begin{eqnarray}
C(t)&=&\frac{e^{|t|(1-\lambda+\mu\Lip(\nabla \vu))}-1}{1-\lambda+\mu\Lip(\nabla \vu)}
C_{\|\cH R^{in}\|_1,\|\vu\|_\infty,\|\vd\|_\infty,\|\nabla \vd\|_\infty}<\infty,\nonumber\\
D(t)&=&e^{|t|(1-\lambda+\mu\Lip(\nabla \vu))}
D_d
<\infty.\nonumber
\end{eqnarray}
\end{Thm}

\begin{Cor}\label{main2mix}
For every $t\ge 0$,
$$
\dis(R_0(t),R_\eps(t))\leq E(t)\epsilon^{\frac13}.
$$

\end{Cor}
In Theorem \ref{main1mix} and Corollary \ref{main2mix}, the constants $C(t),D(t),E(t)$ are the same as in Theorem \ref{main1} and Corollary \ref{main2} after replacing 
$(\psi^{in},\cH\psi^{in})$ by $\|\cH R^{in}\|_1$ in $C(t)$.
\vskip 1cm
\section{Proof of Theorems \ref{main1} and \ref{main1mix}, and Proposition \ref{propdelta}}\label{proofmain}
For all $R,S\in\cD(\fH)$, we denote by $\cC(R,S)$ the set of couplings of $R$ and $S$, i.e.
$$
\cC(R,S):=\{Q\in\cD(\fH\otimes\fH)\text{ s.t. }\Tr_{\fH\otimes\fH}((A\otimes I+I\otimes B)Q)=\Tr_\fH(AR+BS)\}\,.
$$
We recall the definition of the pseudo-distance $MK_\hb$ (see Definition 2.2 in \cite{FGMouPaul}).

\begin{Def}
For each $R,S\in\cD_2(\fH)$, 
$$
MK_\hb(R,S):=\inf_{Q\in\cC(R,S)}\sqrt{\Tr_{\fH\otimes\fH}(Q^{1/2}CQ^{1/2})}\,,
$$
where
$$
C:=\sum_{j=1}^d((q_j\otimes I-I\otimes q_j)^2-\hb^2(\d_{q_j}\otimes I-I\otimes\d_{q_j})^2)\,.
$$
\end{Def}
 Theorem \ref{main1} and Proposition \ref{propdelta} are a consequence of  the following inequality, which controls the continuous dependence of the solution to the von Neumann equation in terms of the initial data and on the potential.

\begin{Thm}\label{main}
Let $R_\eps^{in}\in\cD_2(\fH)$ and $R_\eps(t)$ be the solution of \eqref{vNeps}, $\eps\in[0,1]$. Then, for each $t\in\bR,\eps\in[0,1]$, one has
$$
\ba
MK_\hb(R_0(t),R_\eps(t))^2\le e^{|t|\Lambda(\grad \vu)}MK_\hb(R_0^{in},R_\eps^{in})^2
\\
+\vp\frac{e^{|t|\Lambda(\grad \vu)}-1}{\Lambda(\grad \vu)}\|\grad
\vd\|_{L^\infty}\sqrt{\Tr((R_0^{in})^{1/2}{\cH} (R_0^{in})^{1/2})+2\mu\|\vu\|_{L^\infty}}
\\
+\vp\frac{e^{|t|\Lambda(\grad \vu)}-1}{\Lambda(\grad \vu)}\|\grad
\vd\|_{L^\infty}\sqrt{\Tr((R_\eps^{in})^{1/2}{\cH} (R_\eps^{in})^{1/2})+2(\mu\|\vu\|_{L^\infty}+\vp\|\vd\|_{L^\infty})}&\,,
\ea
$$
where
$$
\Lambda(\grad \vu)=1-\lambda+\mu\Lip(\grad \vu).
$$
\end{Thm}

Note that, as mentioned before, when $\lambda=1$, $\Lambda(\grad \vu)=\mu\Lip(\grad \vu)$ and in the inequality above, the function
$$
z\mapsto\frac{e^z-1}{z}
$$
is extended by continuity at $z=0$.

\begin{proof} In order to lighten the formulas we will 
use the following notations
\be\label{defvuvd}
V_1=\mu\vu,\  V_2=\mu\vu+\eps\vd,
\ee
so that
$$\cH_0=\cH^{\lambda,0}_0+V_1\mbox{ and }\cH_\eps=\cH^{\lambda,0}_0+V_2.
$$
%

Let $Q^{in}\in\cC(R_0^{in},R_\eps^{in})$, and let $Q$ be the solution of the von Neumann equation
$$
i\hb\d_tQ=[({\cH^{\lambda,0}_0}+V_1)\otimes I+I\otimes({\cH^{\lambda,0}_0}+V_2),Q]\,,\quad Q\rstr_{t=0}=Q^{in}\,.
$$
Then
$$
Q(t)\in\cC(R_0(t),R_\eps(t))\,,\quad\text{ for each }t\ge 0
$$
(see for instance Lemma 5.1 in \cite{FGMouPaul}). 

Next we compute
$$
\ba
\frac{d}{dt}\Tr_{\fH\otimes\fH}(Q(t)^{1/2}CQ(t)^{1/2})
\\
=\frac{i}\hb\Tr_{\fH\otimes\fH}(Q(t)^{1/2}[({\mathcal{H}^{\lambda,0}_0}+V_1)\otimes I+I\otimes({\mathcal{H}^{\lambda,0}_0}+V_2),C]Q(t)^{1/2})&\,.
\ea
$$

One finds that
$$
\frac{i}\hb[-\tfrac12\hb^2(\Dlt\otimes I+I\otimes\Dlt),C]=\sum_{j=1}^d(q_j\otimes I-I\otimes q_j)\vee(-i\hb(\d_{q_j}\otimes I-I\otimes\d_{q_j}))\,,
$$
while
$$
\frac{i}\hb[(V_1\otimes I+I\otimes V_2),C]=-\sum_{j=1}^d(\d_{q_j}V_1\otimes I-I\otimes\d_{q_j}V_2)\vee(-i\hb(\d_{q_j}\otimes I-I\otimes\d_{q_j}))\,,
$$
with the notation
$$
A\vee B:=AB+BA\,.
$$
In particular
$$
\frac{i}\hb[\tfrac12(|q|^2\otimes I+I\otimes|q|^2),C]=-\sum_{j=1}^d(q_j\otimes I-I\otimes q_j)\vee(-i\hb(\d_{q_j}\otimes I-I\otimes\d_{q_j}))\,.
$$
See \cite{FGMouPaul} on p. 190. Hence
$$
\ba
\frac{d}{dt}\Tr_{\fH\otimes\fH}(Q(t)^{1/2}CQ(t)^{1/2})
\\
-(1-\lambda)\Tr_{\fH\otimes\fH}\left(Q(t)^{1/2} \sum_{j=1}^d(q_j\otimes I-I\otimes q_j)\vee(-i\hb(\d_{q_j}\otimes I-I\otimes\d_{q_j}))Q(t)^{1/2}\right)\\
=\Tr_{\fH\otimes\fH}\left(Q(t)^{1/2}\sum_{j=1}^d(\d_{q_j}V_1\otimes I-I\otimes\d_{q_j}V_2)\vee(-i\hb(\d_{q_j}\otimes I-I\otimes\d_{q_j}))Q(t)^{1/2}\right)
\\
=\Tr_{\fH\otimes\fH}\left(Q(t)^{1/2}\sum_{j=1}^d(\d_{q_j}V_1\otimes I-I\otimes\d_{q_j}V_1)\vee(-i\hb(\d_{q_j}\otimes I-I\otimes\d_{q_j}))Q(t)^{1/2}\right)
\\
+\Tr_{\fH\otimes\fH}\left(Q(t)^{1/2}\sum_{j=1}^d(I\otimes\d_{q_j}(V_1-V_2)\vee(-i\hb(\d_{q_j}\otimes I-I\otimes\d_{q_j}))Q(t)^{1/2}\right)
\\
=:\tau_1+\tau_2&\,.
\ea
$$

At this point, we recall the elementary operator inequality
$$
A^*B+B^*A\le A^*A+B^*B\,.
$$
Therefore, 
$$
\sum_{j=1}^d(q_j\otimes I-I\otimes q_j)\vee(-i\hb(\d_{q_j}\otimes I-I\otimes\d_{q_j}))
\leq
C
$$
and, for each $\ell>\Lip(\grad V_1)^{1/2}$, one has
$$
\ba
\sum_{j=1}^d(\d_{q_j}V_1\otimes I-I\otimes\d_{q_j}V_1)\vee(-i\hb(\d_{q_j}\otimes I-I\otimes\d_{q_j}))
\\
=\sum_{j=1}^d\frac1\ell(\d_{q_j}V_1\otimes I-I\otimes\d_{q_j}V_1)\vee\ell(-i\hb(\d_{q_j}\otimes I-I\otimes\d_{q_j}))
\\
\le\sum_{j=1}^d\left(\frac{\Lip(\grad V_1)^2}{\ell^2}(q_j\otimes I-I\otimes q_j)^2+\ell^2(-i\hb(\d_{q_j}\otimes I-I\otimes\d_{q_j}))^2\right)&\,.
\ea
$$
Letting $\ell\to\Lip(\grad V)^{1/2}$ shows that
$$
\sum_{j=1}^d(\d_{q_j}V_1\otimes I-I\otimes\d_{q_j}V_1)\vee(-i\hb(\d_{q_j}\otimes I-I\otimes\d_{q_j}))\le\Lip(\grad V_1)C\,,
$$
so that
$$
\tau_1\le\Lip(\grad V_1)\Tr_{\fH\otimes\fH}(Q(t)^{1/2}CQ(t)^{1/2})\,.
$$

For the term $\tau_2$, we simply use the Cauchy-Schwarz inequality:
$$
\ba
\tau_2=\Tr_{\fH\otimes\fH}\left(Q(t)^{1/2}\sum_{j=1}^d(I\otimes\d_{q_j}(V_1-V_2)\vee(-i\hb(\d_{q_j}\otimes I-I\otimes\d_{q_j}))Q(t)^{1/2}\right)
\\
\le\sum_{j=1}^d\|Q(t)^{1/2}(I\otimes\d_{q_j}(V_1-V_2))\|_2||-i\hb(\d_{q_j}\otimes I-I\otimes\d_{q_j})Q(t)^{1/2}||_2
\\
\le\sum_{j=1}^d\|Q(t)^{1/2}(I\otimes\d_{q_j}(V_1-V_2))\|_2||(-i\hb\d_{q_j}\otimes I)Q(t)^{1/2}||_2
\\
+\sum_{j=1}^d\|Q(t)^{1/2}(I\otimes\d_{q_j}(V_1-V_2))\|_2+||(I\otimes(-i\hb\d_{q_j}))Q(t)^{1/2}||_2
\\
=\tau_{21}+\tau_{22}&\,.
\ea
$$
Now
$$
\ba
\tau_{21}\le\left(\sum_{j=1}^d\|Q(t)^{1/2}(I\otimes\d_{q_j}(V_1-V_2))\|^2_2\right)^{1/2}\left(\sum_{j=1}^d\|(-i\hb\d_{q_j}\otimes I)Q(t)^{1/2}\|^2_2\right)^{1/2}
\\
=\left(\sum_{j=1}^d\Tr\left(Q(t)^{1/2}(I\otimes\d_{q_j}(V_1-V_2))^2Q(t)^{1/2}\right)\right)^{1/2}
\\
\times\left(\Tr\left(Q(t)^{1/2}(-\hb^2\Dlt\otimes I)Q(t)^{1/2}\right)\right)^{1/2}
\ea
$$
so that
$$
\tau_{21}\le\|\grad(V_1-V_2)\|_{L^\infty(\bR^d)}\Tr(R_0(t)^{1/2}{\mathcal{H}^{\lambda,0}_0} R_0(t)^{1/2})^{1/2}\,,
$$
and likewise
$$
\tau_{22}\le\|\grad(V_1-V_2)\|_{L^\infty(\bR^d)}\Tr(R_\eps(t)^{1/2}{\mathcal{H}^{\lambda,0}_0} R_\eps(t)^{1/2})^{1/2}\,.
$$
Summarizing, we have proved that
$$
\ba
\frac{d}{dt}\Tr_{\fH\otimes\fH}(Q(t)^{1/2}CQ(t)^{1/2})\le
\Tr_{\fH\otimes\fH}(Q(t)^{1/2}CQ(t)^{1/2})\\
+\Lip(\grad V_1)\Tr_{\fH\otimes\fH}(Q(t)^{1/2}CQ(t)^{1/2})
\\
+\|\grad(V_1-V_2)\|_{L^\infty(\bR^d)}(\Tr(R_0(t)^{1/2}{\mathcal{H}^{\lambda,0}_0} R_0(t)^{1/2})^{1/2}\\
+\Tr(R_\eps(t)^{1/2}{\mathcal{H}^{\lambda,0}_0} R_\eps(t)^{1/2})^{1/2})&\,.
\ea
$$
On the other hand, since
$$
i\hb\d_tR_j(t)=[{\mathcal{H}^{\lambda,0}_0}+V_j,R_j(t)]\,,
$$
one has
$$
\Tr(R_j(t)^{1/2}({\mathcal{H}^{\lambda,0}_0}+V_j)R_j(t)^{1/2})=\Tr((R_j^{in})^{1/2}({\mathcal{H}^{\lambda,0}_0}+V_j)(R_j^{in})^{1/2})
$$
so that
$$
\Tr(R_j(t)^{1/2}{\mathcal{H}^{\lambda,0}_0} R_j(t)^{1/2})\le\Tr((R_j^{in})^{1/2}{\mathcal{H}^{\lambda,0}_0} (R_j^{in})^{1/2})+2\|V_j\|_{L^\infty(\bR^d)}\,.
$$
Hence
$$
\ba
\frac{d}{dt}\Tr_{\fH\otimes\fH}(Q(t)^{1/2}CQ(t)^{1/2})\le(1-\lambda+\Lip(\grad V_1))\Tr_{\fH\otimes\fH}(Q(t)^{1/2}CQ(t)^{1/2})
\\
+\|\grad(V_1-V_2)\|_{L^\infty(\bR^d)}\sqrt{\Tr((R_0^{in})^{1/2}{\mathcal{H}^{\lambda,0}_0} (R_0^{in})^{1/2})+2\|V_1\|_{L^\infty(\bR^d)}}
\\
+\|\grad(V_1-V_2)\|_{L^\infty(\bR^d)}\sqrt{\Tr((R_\eps^{in})^{1/2}{\mathcal{H}^{\lambda,0}_0} (R_\eps^{in})^{1/2})+2\|V_2\|_{L^\infty(\bR^d)}}&\,.
\ea
$$
By Gronwall's inequality, choosing $Q^{in}$ to be an optimal coupling of $R_0^{in}$ and $R_\eps^{in}$, one finds that, denoting $\Lambda(\grad V_1)=1-\lambda+\Lip(\grad V_1)$,
$$
\ba
MK_\hb(R_0(t),R_\eps(t))^2\le\Tr_{\fH\otimes\fH}(Q(t)^{1/2}CQ(t)^{1/2})\le e^{t\Lambda(\grad V_1)}MK_\hb(R_0^{in},R_\eps^{in})^2
\\
+\frac{e^{t\Lambda(\grad V_1)}-1}{\Lambda(\grad V_1)}\|\grad(V_1-V_2)\|_{L^\infty(\bR^d)}\sqrt{\Tr((R_0^{in})^{1/2}{\mathcal{H}^{\lambda,0}_0} (R_0^{in})^{1/2})+2\|V_1\|_{L^\infty(\bR^d)}}
\\
+\frac{e^{t\Lambda(\grad V_1)}-1}{\Lambda(\grad V_1)}\|\grad(V_1-V_2)\|_{L^\infty(\bR^d)}\sqrt{\Tr((R_\eps^{in})^{1/2}{\mathcal{H}^{\lambda,0}_0}{\mathcal{H}^{\lambda,0}_0} (R_\eps^{in})^{1/2})+2\|V_2\|_{L^\infty(\bR^d)}}&\,,
\ea
$$
which is the desired inequality by coming back to $\vu,\vd$ through \eqref{defvuvd} and using
$$
\|\mu\vu+\eps\vd\|\leq \mu\|\vu\|+\eps\|\vd\|. 
$$ 

\end{proof}

\begin{proof}[Proof of Theorem \ref{main1} and Proposition \ref{propdelta}]
Observe that the estimate above is uniform in $\hb$ --- more precisely, the moduli of continuity in the initial data and in the potential are independent of $\hb$. Of course, the pseudo-distance $MK_\hb$ itself is not independent of $\hb$.
For $R,S\in\cD(\fH)$ let us define
\begin{eqnarray}\label{defdelta}
&&\delta(W_\hb[R],W_\hb[S]):=\\
&&
\sup_{\substack{
\max\limits_{|\a|+|\b|\leq 2[d/4]+3}\|\d^\a_q\d^\b_pf\|_{L^\infty(\bR^{2d})}\leq 1
}}
|\int_{\bR^{2d}}f(q,p)(W_\hb[R],W_\hb[S])*q,p)dqdp|.\nonumber
\end{eqnarray}
\begin{Lem}\label{lemtheor}
For any $R,S\in\cD_2(\fH)$,
$$
\dis(R,S)\leq 2^d\delta(W_\hb[R],W_\hb[S])\leq 2^d(MK_\hb(R,S)+C_d\hbar)\ \mbox{ with }\ C_d=(1+\frac{\gamma_d}{\sqrt\pi})2d,
$$
where $\gamma_d\leq \frac{d^{3/4}(192e^{-\frac14}\pi^{-\frac54})^{ d}}{4e^{\frac14} }(d^d)^{11/4}$ is the constant appearing in the Caldaron-Vaillancourt theorem (see Appendix C in \cite{GPsemic}).
\end{Lem}
\begin{proof}
The proof consists in applying Theorem A.7 in \cite{FGJinPaul} and Theorem 2.3 (2) in \cite{FGMouPaul}.\end{proof}

\begin{color}{black}
Using Lemma \ref{lemtheor} and Theorem \ref{main} we get  that
$$
\begin{aligned}
(2^{-d}\dis(R_0(t),R_\eps(t))-C_d\hbar)^2\\
\leq (\delta(W_\hb[R_0(t)],W_\hb[R_\eps(t)])-C_d\hbar)^2\\
\leq MK_\hb(R_0(t),R_\eps(t))^2\\
\le e^{t\Lambda(\vp\grad V_1)}MK_\hb(R_0^{in},R_\eps^{in})^2
\\
+\vp\frac{e^{t\Lambda(\vp\grad V_1)}-1}{\Lambda(\grad V_1)}\|\grad(V_1-V_2)\|_{L^\infty(\bR^d)}\sqrt{\Tr((R_0^{in})^{1/2}{\mathcal{H}^{\lambda,0}_0} (R_0^{in})^{1/2})+2\vp\|V_1\|_{L^\infty(\bR^d)}}
\\
+\vp\frac{e^{t\Lambda(\vp\grad V_1)}-1}{\Lambda(\grad V_1)}\|\grad(V_1-V_2)\|_{L^\infty(\bR^d)}\sqrt{\Tr((R_\eps^{in})^{1/2}{\mathcal{H}^{\lambda,0}_0} (R_\eps^{in})^{1/2})+2\vp\|V_2\|_{L^\infty(\bR^d)}}&\,
\end{aligned}
$$
\begin{eqnarray}
&&:=e^{t\Lambda(\grad V_1)}MK_\hb(R_0^{in},R_\eps^{in})^2
+\gamma(t)\vp.\label{defgam}
\end{eqnarray}
 Therefore Theorem \ref{main1} and Proposition  \ref{propdelta} are proven as soon as
 \be\label{condin}
MK_\hb(R^{in},R^{in})^2\leq D'\hbar
\ee
since then
\begin{eqnarray}
\dis(R_0(t),R_\eps(t))&\leq&2^d\delta(W_\hb[R_0(t)],W_\hb[R_\eps(t)])\nonumber\\
&\leq&
\sqrt{2^{2d}e^{|t|\mu\Lip{(\nabla\vu)}}D'\hbar+2^{2d}\gamma(t)\eps}+2^{d}C_d\hbar\nonumber\\
&\leq&
\sqrt{C(t)\eps+D(t)\hbar}\nonumber
\end{eqnarray}
with, since $\hbar\leq 1$,
\begin{eqnarray}
C(t)&=&2^{2d+1}\gamma(t)
\label{cdet}\\
D(t)&=& e^{|t|\mu\Lip{(\nabla\vu)}}2^{2d+1}(D'(t)
+C_d^2)\label{ddet}
\end{eqnarray}
\end{color}
\vskip 1cm
\begin{Lem}\label{lemdelta}
For $R\in\cD_2(\fH)$ let
$$
\Delta(R):=\sqrt{\Tr{\big(R^\frac12(((x-\Tr{(R^\frac12 xR^\frac12)})^2
+
(-i\nabla_x-\Tr{(R^\frac12 (-i\nabla_x)R^\frac12)})^2)R^\frac12\big)}
}.
$$
Note that when $R$ is a pure state i.e. $R=|\psi\rangle\langle\psi|$, $\Delta(R)$ is, modulo a slight abuse of notation, the same as in the definition \eqref{defdeltapsi}.

Then
$$
MK_\hb(R,R)\leq \sqrt2\Delta(R).
$$
Moreover, for any $\psi\in\fH$,
$$
MK_\hb(|\psi\rangle\langle\psi|,|\psi\rangle\langle\psi|)= \sqrt2\Delta(|\psi).
$$
\end{Lem}
\begin{proof}
The  proof consists in remarking that $R\otimes R$ is indeed a coupling between $R$ and itself.  Therefore
\begin{eqnarray}
MK_\hb(R)^2
&\leq&
\Tr{\big((R^\frac12\otimes R^\frac12) C(R^\frac12\otimes R^\frac12)\big)}\nonumber\\
&=&
\Tr{\big((R^\frac12\otimes R^\frac12) (x^2\otimes I+I\otimes x^2-2x\otimes x)(R^\frac12\otimes R^\frac12)\big)}\nonumber\\
&&+\mbox{ same with }x\leftrightarrow -i\hbar\nabla_x\nonumber\\
&=&
\Tr{\big(R^\frac12 x^2 R^\frac12+R^\frac12x^2R^\frac12-2 R^\frac12xR^\frac12.R^\frac12xR^\frac12\big)}\nonumber\\
&&+\mbox{ same with }x\leftrightarrow -i\hbar\nabla_x\nonumber\\
&=&
2\Tr{\big(R^\frac12 x^2 R^\frac12- (R^\frac12xR^\frac12)^2\big)}\nonumber\\
&&+\mbox{ same with }x\leftrightarrow -i\hbar\nabla_x\nonumber\\
&=&2\Delta(R).\nonumber
\end{eqnarray}
The equality is proven the same way, after Lemma 2.1 (ii) in \cite{GParma} which stipulates that the only coupling between $|\psi\rangle\langle\psi|$ and itself is $|\psi\rangle\langle\psi|\otimes|\psi\rangle\langle\psi|$.
\end{proof}
This proves Theorem \ref{main1}, and Theorem \ref{main1mix} when $R^{in}$ satisfies  hypothesis \ref{idelta}. 

%
%

If the initial data $R^{in}$ is a T\"oplitz operator, specifically if
$$
R^{in}=\Op^T_\hb[(2\pi\hb)^d\mu^{in}]\,,
$$
with the notation of \cite{FGMouPaul} one can go further and apply Theorem 2.3 (1) in \cite{FGMouPaul}:
$$
MK_\hb(R^{in},R^{in})\leq 2d\hbar,
$$
so that \eqref{condin} is again satisfied and Theorem \ref{main1mix} is proven when $R^{in}$ satisfies the hypothesis \ref{itoplitz}.

The proof in the case of the  hypothesis \ref{iwigner}, \ref{ihermitte} and \ref{itoplitz2} follows directly the first inequality of Theorem 8.1 in \cite{GParma} with $R=S=R^{in}$, together with item $(I)$ (through the Corollary of Theorem 3.1 in \cite{GPsemic}) and item $(II)$ of Theorem 4.1 in \cite{GPsemic}, with $\mu(\hbar)=\mu'(\hbar)=C,\ \nu(\hbar)=\nu'(\hbar)=C\sqrt\hb\mbox{ and } \tau(\hbar)=\hbar$, respectively. 

Indeed, \cite[Theorem 8.1, (iii) first inequality]{GParma} 
stipulates that, for all density matrix $R$,
$$
MK_\hb(R^{in},R^{in})\leq 2 E_\hb(\widetilde W_\hb[{R^{in}}],R^{in})
$$
where $E_\hb$ is a semiquantum pseudometric whose knowledge of the definition \cite[Definition 2.2]{FGPaul} is not strictly necessary for our purpose here since Theorem 4.1 in \cite{GPsemic}) shows that, when $\mu(\hbar)=\mu'(\hbar)=C,\ \nu(\hbar)=\nu'(\hbar)=C\sqrt\hb\mbox{ and } \tau(\hbar)=\hbar$,
$$
E_\hb(\widetilde W_\hb[{R^{in}}],R^{in})=O(\sqrt{\hbar}).
$$
Hence
$$
MK_\hb(R^{in},R^{in})^2\leq D'\hbar
$$
for some constant $D'$ explicitely recoverable from  \cite[Theorem 4.1]{GPsemic}.

This completes the proof of Theorems \ref{main1} and \ref{main1mix}.\end{proof}

\section{Proof of Corollaries \ref{main2} and \ref{main2mix}}\label{proffcor}
Let us first derive the easy standard following estimate.

\begin{Prop}\label{proptrac}
For every $t\in\bR$,
$$
\|R_0(t)-R_\eps(t)\|_1\leq 2t\frac\vp{\hbar}\|\vd\|_\infty
$$
\end{Prop}
\begin{proof}
The solution of \eqref{vNeps} is explicitly given by
$$
R_\eps(t)=e^{-i\frac{t\cH_\eps}\hb}R^{in}e^{i\frac{t\cH_\eps}\hb}
$$
Therefore, one easily shows that
$$
R_0(t)-R_\eps(t)=\frac1{i\hbar}\int_0^t
e^{-i\frac{(t-s)\cH_\eps}\hb}[\eps\vd,R_0(t)]e^{i\frac{(t-s)\cH_\eps}\hb}.
$$ 
and the result follows from
$$
\|[\eps\vd,R_0(t)]\|_1\leq 2\eps\|\vd\|\|R_0(t)
\|_1.
$$
\end{proof}
Corollaries \ref{main2},\ \ref{main2mix} will follow by interpolation between Theorem \ref{main} and Proposition \ref{proptrac}, through the following inequality.

\begin{Lem}\label{lema7}
For any $R,S\in\cD_2(\fH)$,
$$
\dis(R,S)\leq 
2^{d}\|R-S\|_1
$$
\end{Lem}
\begin{proof}
This is  Theorem A7 in \cite{FGJinPaul}, item $(2)$ 
\end{proof}

Therefore, by Theorem \ref{main1} and Lemma \ref{lema7} we get, for each $\hbar,\vp\in[0,1]$,
$$
\dis(R_1(t),R_2(t))
\leq
\min{(\sqrt{C(t)\vp+D(t)\hbar},2|t|\frac\vp\hb\|\vd\|_\infty)}.
$$
Obviously, 
$\min{(\sqrt{C(t)\vp+D(t)\hbar},2|t|\frac\vp\hb\|\vd\|_\infty)}\leq \min{(\sqrt{C(t)+D(t)},2|t|\|\vd\|_\infty)}\eps^\frac13$ for $\vp,\hbar\leq 1$,
since, when $\hbar\leq\vp^\frac23$, $\hbar,\eps\leq\eps^\frac23$, and, when $\hbar\geq\vp^\frac23$, $\frac\vp\hb\leq\vp^\frac13$.
The Corollary is proved.

\section{Proof of Proposition \ref{crasprop}}\label{proofcrasprop}

The upper bound  is given simply by Lemma \ref{lemtheor} and the following inequality, proved in \cite[Section 4]{gpcoh}
$$
MK_\hb(|\psi_{z_1}\rangle\langle\psi_{z_1}|,|\psi_{z_2}\rangle\langle\psi_{z_2}|)^2\leq |z_1-z_2|^2+2d\hbar.
$$

For the lower bound, we will pick  a test operator in the form of a T\" oplitz operator with symbol $f\geq 0$:
$$
F:=\Op^T_\hb[(2\pi\hb)^df].
$$
One easily verifies that $\Tr{F}=\int_{\bR^{2d}}f(q,p)dqdp$.

Moreover, see \cite[Appendix B]{FGJinPaul},
\begin{eqnarray}
\tfrac1{i\hbar}[F,-i\hbar\partial_{x_j}]&=&
\Op^T_\hb[(2\pi\hb)^d\partial_{p_j}f],\ j=1,\dots,d,\nonumber\\
\tfrac1{i\hbar}[F,x_j]&=&
\Op^T_\hb[(2\pi\hb)^d(-\partial_{q_j}f)],\ j=1,\dots,d,\nonumber\\
\|F\|_{1}&\leq&\|f\|_{L^1(\bR^{2d})}.\nonumber
\end{eqnarray}
Therefore, it is easy to construct functions $f$ such that $F:=\Op^T_\hb[(2\pi\hb)^df]$ satisfies the constraints 
of the maximization problem in the definition of  $\dis$.

Moreover, denoting $z=(q,p)\in\bR^{2d}$,
\begin{eqnarray}
\dis(|\psi_{z_1}\rangle\langle
\psi_{z_1}|,|\psi_{z_2}\rangle\langle\psi_{z_2}|)&\geq&
|\Tr{\left(F(|\psi_{z_1}\rangle\langle
\psi_{z_1}|-|\psi_{z_2}\rangle\langle\psi_{z_2}|)\right)}|\nonumber\\&=&|\langle\psi_{z_1}|F|\psi_{z_1}\rangle
-
\langle\psi_{z_2}|F|\psi_{z_2}\rangle|\nonumber\\
&=&
|\int_{\bR^{2d}}f(z)(|\langle\psi_z|\psi_{z_1}\rangle|^2
-
|\langle\psi_z|\psi_{z_2}\rangle|^2)\tfrac{dqdp}{(2\pi\hbar)^d}|\nonumber\\
&=&
|\int_{\bR^{2d}}f(z)
(e^{-\frac{|z-z_1|^2}{2\hbar}}-e^{-\frac{|z-z_2|^2}{2\hbar}})
\tfrac{dqdp}{(2\pi\hbar)^d}|\nonumber\\
&\geq& |f(z_1)-f(z_2)|\nonumber\\
&-&\tfrac2e{\hbar}\max_{\|\a|,|\b|\leq 2}\int_{\bR^{2d}}|\partial_q^\a\partial_p^\b f(q,p)|dqdp.\nonumber
\end{eqnarray}

Let us suppose now that $f\in\cS(\bR^{2d})$ and $f$ is convex in a convex  domain containing $z_1,z_2$. Then, one can certainly rescale, translate and rotate 
$f$ such that
\begin{itemize}
\item 
$f(z_1)-f(z_2)\geq \nabla f(z_2).(z_1-z_2)$ \hfill{convexity}
\item 
$|f(z_1)-f(z_2)|\geq C|z_1-z_2|,\  C>0$\hfill{rotation and translation}
\item
$\max\limits_{\substack{|\a|,|\b| \le 2[\frac d4]+3}}\|\mathcal D^\a_{-i\hbar\nabla}\mathcal D^\b_{x}F\|_{1}\leq \max\limits_{\|\a|,|\b|\leq 2}\|\partial_q^\a\partial_p^\b f\|_{L^1(\bR^{2d})}\leq 1$\hfill{rescaling.}
\end{itemize}
Hence
$$
C|z_1-z_2|-\hbar\leq 
\dis(|\psi_{z_1}\rangle\langle
\psi_{z_1}|,|\psi_{z_2}\rangle\langle\psi_{z_2}|).
$$


\begin{thebibliography}{99}

\bibitem[AFFGP10]{AFFGP} { L.Ambrosio, A. Figalli, G. Friesecke, J. Giannoulis, T. Paul :}
{\em Semiclassical limit of quantum dynamics with rough potentials and well posedness of transport equations with measure initial data.} arXiv:1006.5388v1, (2010).

\bibitem[B28]{Bi} G. D. Birkhoff, ``Dynamical systems", 
American Mathematical Society Colloquium Publications, Vol. IX American Mathematical Society, Providence, R.I. (1966). 


\bibitem[B25]{B} M. Born, ``Vorlesungen \"uber Atommechanik", Springer, Berlin, (1925). English translation: ``The
mechanics of the atom", Ungar, New-York, (1927).

\bibitem[D91]{DGH}	
M. Degli Esposti, S. Graffi, J. Herczynski,  {\it Quantization of the classical Lie algorithm in the Bargmann representation}, Annals of Physics, {\bf 209} 2 (1991), 364-392.




\bibitem[FLP13]{filipa} A. Figalli, M. Ligab\`o, T. Paul, {\it Semiclassical limit for mixed states
with singular and rough potentials,  Indiana University Mathematics Journal, \textbf{61}, (2013), 193-222.\smallskip\noindent}

\bibitem[GJP20]{FGJinPaul} F. Golse, S. Jin, T. Paul: \textit{On the convergence of the time splitting methods for quantum dynamics},  Foundations of Computational Mathematics · DOI: 10.1007/s10208-020-09470-z (2020).\smallskip\noindent

\bibitem[GMP16]{FGMouPaul}
F. Golse, C. Mouhot, T. Paul:
\textit{On the Mean Field and Classical Limits of Quantum Mechanics},
Commun. Math. Phys. \textbf{343} (2016), 165--205.

\bibitem[GP17]{FGPaul}
F. Golse, T. Paul:
\textit{The Schr\"odinger Equation in the Mean-Field and Semiclassical Regime},
Arch. Rational Mech. Anal. \textbf{223} (2017), 57--94.

\bibitem [GP18]{gpcoh}F. Golse,  T. Paul:
\textit{Wave Packets and the Quadratic Monge-Kantorovich Distance in Quantum Mechanics}, Comptes Rendus Mathematique \textbf{356} 177 - 197 (2018).\smallskip\noindent

\bibitem[GP20a]{GPsemic}
F. Golse,  T. Paul:
\textit{Semiclassical evolution with low regularity }, J. Math. Pures et Appl. (2021) DOI:10.1016/j.matpur.2021.02.008


\bibitem[GP20b]{GParma}
F. Golse,  T. Paul:
\textit{Quantum and Semiquantum Pseudometrics and applications}, preprint arXiv:2102.05184 [math.AP].


\bibitem[G87]{GP1} S. Graffi, T.Paul, {\it Schr\"odinger equation and canonical perturbation theory}, Comm. Math. Phys., {\bf 108}, 25-40 (1987).

\bibitem[H25]{H} W. Heisenberg, {\it Matrix mechanik}, Zeitscrift f\"ur Physik, {\bf 33}, 879-893 (1925).

\bibitem[CJ10]{jabin}N. Champagnat, P-E. Jabin : {\em Well Posedness in any Dimension for Hamiltonian
Flows with Non BV Force Terms}, Communications in Partial Differential Equations {\bf 35} (2010), 786-816

\bibitem[LP93]{lionspaul} { P.L.Lions, T.Paul:} {\em Sur les mesures de Wigner.}
Rev. Mat. Iberoamericana {\bf 9} (1993), 553--618.


\bibitem[NPST18]{NPST}  J-C. Novelli, T. Paul, D. Sauzin, J-Y. Thibon: {\it Rayleigh-Schr\"odinger series and Birkhoff decomposition},  Letters in Mathematical Physics \text{108} 1583-1600 (2018).\smallskip\noindent

\bibitem[P16]{Dyn} T. Paul, D. Sauzin, {\it Normalization in Lie algebras via mould calculus and applications},  Regular and Chaotic Dynamics \textbf{22} (6) 616-649 (2017), special issue in memory of  Vladimir Arnold.\smallskip\noindent.

\bibitem[P162]{Dyn2} T. Paul, D. Sauzin, {\it Normalization in Banach scale Lie algebras via mould calculus and applications},   Discrete and Continuous Dynamical Systems A \textbf{37}  4461 - 4487 (2017).\smallskip\noindent


\bibitem[P1892]{P} H. Poincar\'e, ``Les m\'ethodes nouvelles de la m\'ecanique 
c\'eleste", Volume 2, Gauthier-Villars, Paris, (1892), Blanchard, Paris, (1987).
\end{thebibliography}
\end{document}